\documentclass{aims}
\usepackage{amsmath}
  \usepackage{paralist}
  \usepackage{graphics} %% add this and next lines if pictures should be in esp format
  \usepackage{epsfig} %For pictures: screened artwork should be set up with an 85 or 100 line screen
\usepackage{graphicx}  \usepackage{epstopdf}%This is to transfer .eps figure to .pdf figure; please compile your paper using PDFLeTex or PDFTeXify.
 \usepackage[colorlinks=true]{hyperref}
\hypersetup{urlcolor=blue, citecolor=red}

\def\currentvolume{X}

    \def\ppages{X--XX}

\newtheorem{theorem}{Theorem}[section]
\newtheorem{corollary}{Corollary}

\newtheorem{lemma}[theorem]{Lemma}

\theoremstyle{definition}

\newtheorem{remark}{Remark}

\usepackage{mathtools}
\usepackage{algpseudocode}
\usepackage{algorithm}

\newcommand{\gamD}{\Gamma^\textup{D}}
\newcommand{\gamN}{\Gamma^\textup{N}}
\newcommand{\arum}{\ensuremath{\mathcal{A}}}
\newcommand{\rum}[1]{\mathbb{#1}}
\newcommand{\po}{{\partial\Omega}}
\newcommand{\fob}{\mathcal{F}}

\newcommand{\hmhalf}{H_\diamond^{-1/2}(\partial\Omega)}
\newcommand{\ds}{\eta}
\newcommand{\dsr}{\delta\sigma}
\newcommand{\imbed}{\hookrightarrow}
\newcommand{\proja}{\mathcal{P}_{\arum_0}}
\newcommand{\ssc}[1]{\ensuremath{\mathcal{S}_{#1}}}
\newcommand{\sref}[1]{Section ~\ref{#1}}

\DeclareMathOperator*{\argmin}{argmin}
\DeclareMathOperator{\sign}{sgn}
\DeclareMathOperator*{\supp}{supp}

\DeclarePairedDelimiter{\norm}{\lVert}{\rVert} 
\DeclarePairedDelimiter{\abs}{\lvert}{\rvert} 
\DeclarePairedDelimiter{\inner}{\langle}{\rangle} 
\newcommand{\absm}[1]{\abs{#1}}
\newcommand{\normm}[1]{\norm{#1}}
\newcommand{\innerm}[1]{\inner{#1}}

\title[3D Reconstruction for Partial Data EIT using a Sparsity Prior] %Use the shortened version of the full title
      {3D Reconstruction for Partial Data Electrical Impedance Tomography Using a Sparsity Prior}

\author[H. Garde and K. Knudsen]{}


\thanks{The authors are supported by ERC project High-Definition Tomography, Advanced Grant No. 291405.}

\begin{document}
\maketitle

% Enter the first author's name and address:
\centerline{\scshape Henrik Garde and Kim Knudsen}
\medskip
{\footnotesize
% please put the address of the first author
 \centerline{Department of Applied Mathematics and Computer Science}
   \centerline{Technical University of
    Denmark}
   \centerline{2800 Kgs. Lyngby, Denmark}
} % Do not forget to end the {\footnotesize by the sign }

%\medskip

%\centerline{\scshape First-name2 last-name2 and First-name3
%last-name3}
%\medskip
%{\footnotesize
% % please put the address of the second  and third author
% \centerline{ First line of the address of the second author}
%   \centerline{Other lines}
%   \centerline{Springfield, MO 65810, USA}
%}

\bigskip

% The name of the associate editor will be entered by an editorial staff
% "Communicated by the associate editor name" is not needed for special issue.

%The abstract of your paper
\begin{abstract}
  \noindent In electrical impedance tomography the electrical
  conductivity inside a physical body is computed from electro-static
  boundary
  measurements. % the mathematical model is an inverse boundary value problem.
  The focus of this paper is to extend recent result for the 2D
  problem to 3D. Prior information about the sparsity and spatial distribution of the conductivity is used to improve
  reconstructions for the partial data problem with Cauchy
  data measured only on a subset of the boundary. A sparsity prior is enforced
  using the $\ell_1$ norm in the penalty term of a Tikhonov functional, and spatial
  prior information is incorporated by applying a spatially
  distributed regularization parameter. The optimization problem is
  solved numerically using a generalized conditional gradient method
  with soft thresholding. Numerical examples show the effectiveness of
  the suggested method even for the partial data problem with
  measurements affected by noise.%the effectiveness of applying
 % prior information to vastly improve reconstruction of inclusions in  3D.
\end{abstract}

%The title of your section 1
\section{Introduction}

Sparse reconstruction for electrical impedance tomography (EIT) with full boundary data has been utilized in \cite{gehre2012,jin2012,jin2011} and are based on algorithms from \cite{bonesky2007,bredies2009}. A similar approach was used in 2D for the partial data problem in \cite{GardeKnudsen2014} with use of a spatially varying regularization parameter, and this paper extends the algorithm to the 3D partial data problem. The main contributions are in deriving the Fr\'{e}chet derivative for the algorithm, and in the numerical results in 3D.

The inverse problem in EIT consists of reconstructing an electrical
conductivity distribution in the interior of an object from
electro-static boundary measurements on the surface of the object. The
underlying mathematical problem is known as the Calder\'on problem in
recognition of Calder\'on's seminal paper \cite{Calderon1980}. While
the Calder\'on problem can also be considered in two dimensions,
physical electric fields are intrinsically three dimensional, and thus
the reconstruction problem in EIT should ideally use a 3D reconstruction
algorithm to reduce modelling errors in the reconstruction.

Consider a bounded domain $\Omega\subset \mathbb{R}^3$ with smooth boundary $\po.$ In
order to consider partial boundary measurements we introduce the
subsets $\gamN,\gamD \subseteq \po$ for the Neumann and Dirichlet data
respectively. Let $\sigma \in L^\infty(\Omega)$ with $0< c\leq \sigma$ a.e.
denote the conductivity distribution in $\Omega$. Applying a boundary
current flux $g$ (Neumann condition) through $\gamN\subseteq \po$
gives rise to the interior electric potential $u$ characterized as the
solution to
\begin{equation}
  \nabla \cdot(\sigma\nabla u) = 0 \text{ in } \Omega,\quad	\sigma \frac{\partial u}{\partial \nu} = g \text{ on } \partial\Omega, \quad \int_{\gamD} u|_{\po}\,ds = 0, \label{pde}
\end{equation}
where $\nu$ is an outward unit normal to $\po.$ The latter condition
in \eqref{pde} is a grounding of the total electric potential along the
subset $\gamD\subseteq \po.$ To be precise we define the spaces 
\begin{align*}
  L_\diamond^2(\po) &\equiv \{g \in L^2(\partial\Omega) \mid \int_{\po}
  g \,ds = 0\}, \\ 
  \hmhalf &\equiv \{g \in H^{-1/2}(\partial\Omega)
  \mid \inner{g,1} = 0\},
\end{align*}
consisting of boundary functions with mean zero, and the spaces
\begin{align*}
  H_{\gamD}^1(\Omega) &\equiv \{u\in H^1(\Omega)\mid u|_{\po} \in
  H_{\gamD}^{1/2}(\po)\,\}, \\ 
  H_{\gamD}^{1/2}(\po) &\equiv \{f\in H^{1/2}(\po)\mid \int_{\gamD} f\,ds=0\},
\end{align*}
consisting of functions with mean zero on $\gamD$ designed to
encompass the partial boundary data.  Using standard elliptic theory
it follows that \eqref{pde} has a unique solution $u \in
H_{\gamD}^1(\Omega)$ for any $g \in \hmhalf$. This defines the Neumann-to-Dirichlet map (ND-map) $\Lambda_\sigma : H_\diamond^{-1/2}(\po) \to H_{\gamD}^{1/2}(\po)$ by $\Lambda_\sigma g = u|_{\po}$, and the partial ND-map as $(\Lambda_\sigma g)|_{\gamD}$ for $\supp{g} \subseteq \gamN.$

Recently the partial data Calder\'on problem has been studied
intensively. In 3D uniqueness has been proved under certain conditions
on $\gamD$ and $\gamN$ \cite{BukhgeimUhlmann2002, Isakov2007, KenigSjostrandUhlmann2007, Knudsen2006}. %   in particular for the geometry applied in the
% numerical examples in \sref{sec:NumericalResults}
% \cite{Isakov2007}.
Also stability estimates of log-log type have been
obtained for the partial problem \cite{HeckWang2006}; this suggests
that the partial data problem is even more ill-posed and hence
requires more regularization than the full data problem which has log
type estimates \cite{Alessandrini1988}.

The data considered here consist of $K$ pairs of Cauchy data taken on the subsets $\gamD$ and $\gamN,$ i.e.
\begin{equation}
	\{(f_k,g_k) \mid g_k \in \hmhalf, \; \supp(g_k)\subseteq \gamN, f_k =  \Lambda_{\sigma} g_k|_{\gamD} \}_{k=1}^K. \label{data}
\end{equation}
We assume that the true conductivity is given as $\sigma = \sigma_0 + \delta\sigma$, where $\sigma_0$ is a known background conductivity. Define the closed and convex subset 
\begin{equation}
	\arum_0 \equiv \{\delta\gamma \in H_0^1(\Omega)\mid c\leq \sigma_0+\delta\gamma\leq c^{-1} \text{  a.e. in } \Omega\}  \label{a0ref}
\end{equation}
for some $c \in (0,1)$, and $\sigma_0\in H^1(\Omega)$ where $c\leq \sigma_0\leq c^{-1}$. Similarly define 
\[
	\arum\equiv \arum_0 + \sigma_0 = \{\gamma\in H^1(\Omega)\mid c\leq \gamma\leq c^{-1} \text{ a.e. in } \Omega, \gamma|_{\po} = \sigma_0|_{\po}\}.
\]
The inverse problem is then to approximate $\delta\sigma\in\arum_0$ given the data \eqref{data}. 
% i.e. we seek to find some change in the conductivity and it is assumed
% to be a small inclusion or multiple small inclusions that are
% compactly supported in the interior of $\Omega$. 

Let $\{\psi_j\}$ denote a chosen orthonormal basis for $H_0^1(\Omega).$ For sparsity regularization we approximate $\delta\sigma$ by $\argmin_{\delta\gamma\in\arum_0}\Psi(\delta\gamma)$ using the following Tikhonov functional
\begin{equation}
	\Psi(\delta\gamma) \equiv \sum_{k=1}^K R_k(\delta\gamma) + P(\delta\gamma), \enskip\delta\gamma\in\arum_0, \label{psieq}
\end{equation} 
with the discrepancy terms $R_k$ and penalty term $P$ given by 
\[
	R_k(\delta\gamma) \equiv \frac{1}{2}\normm{\Lambda_{\sigma_0+\delta\gamma}g_k - f_k}_{L^2(\gamD)}^2, \quad P(\delta\gamma) \equiv \sum_{j=1}^\infty \alpha_j\absm{c_j},
\]
for $c_j \equiv \inner{\delta\gamma,\psi_j}$. The regularization parameter $\alpha_j$ for the
sparsity-promoting $\ell_1$ penalty term $P$ is distributed such that
each basis coefficient can be regularized differently; we will return to
this in \sref{sec:prior}. It should be noted how easy and natural the
use of partial data is introduced in this way, simply by only
minimizing the discrepancy on $\gamD$ where the Dirichlet data is known and ignoring the rest of the boundary.

\begin{remark}
	The non-linearity of $\sigma\mapsto\Lambda_\sigma$ leads to a non-convex discrepancy term, i.e. $\Psi$ is non-convex. When applying a gradient based optimization method, the best we can hope is to find a local minimum.
\end{remark}

This paper is organised as follows: in \sref{sec:SparseReconstruction}
we derive the Fr\'echet derivative of $R_k$ and reformulate the
optimization problem using the generalized conditional gradient method as a
sequence of linearized optimization problems. In \sref{sec:prior} we
explain the idea of the spatially dependent regularization parameter
designed for the use of prior information. Finally, in
\sref{sec:NumericalResults} we show the feasibility of the algorithm
by a few numerical examples.

\section{Sparse Reconstruction}\label{sec:SparseReconstruction}

In this section the sparse reconstruction of $\dsr$ based on the optimization problem \eqref{psieq}, is investigated for a bounded domain $\Omega\subset\rum{R}^3$ with smooth boundary. The penalty term emphasizes that $\dsr$ should only be expanded by few basis functions in a given orthonormal basis. The partial data problem comes into play in the discrepancy term, in which we only fit the data on part of the boundary. Ultimately, this leads to Algorithm \ref{alg1} at the end of this section. 

For fixed $g$ let $u$ be the unique solution to \eqref{pde}. Define
the solution operator $F_g:\sigma\mapsto u$ and further its trace
$\fob_g: \sigma \mapsto u|_{\partial \Omega}$ (note that
$\Lambda_\sigma g = \fob_g(\sigma)$). In order to compute the
derivative of $\fob_g,$ let $\gamma\in\arum$ and $g\in
L^p(\po)\cap\hmhalf$ for $p\geq \frac{8}{5}$. Then following the
proofs of Theorem 2.2 and Corollary 2.1 in \cite{jin2011} whilst
applying the partial boundary $\gamD$ we have
\begin{equation}
	\lim_{\substack{\normm{\ds}_{H^1(\Omega)}\to 0 \\ \gamma+\ds\in\arum}}\frac{\normm{\fob_g(\gamma+\ds)-\fob_g(\gamma) - (\fob_g)'_\gamma\ds}_{H^{1/2}_{\gamD}( \partial\Omega)}}{\normm{\ds}_{H^1(\Omega)}} = 0. \label{fpplimit}
\end{equation}
The linear map $(\fob_g)'_\gamma$ maps $\eta$ to $w|_{\po}$, where $w$ is the unique solution to
\begin{equation}
	-\nabla\cdot(\gamma\nabla w) = \nabla\cdot(\ds\nabla F_g(\gamma)) \text{ in } \Omega, \quad \sigma\frac{\partial w}{\partial\nu} = 0 \text{ on } \partial\Omega, \quad \int_{\gamD}w|_{\partial\Omega}\,ds = 0. \label{fprimepde}
\end{equation}
Note that $(\fob_g)'_\gamma$ resembles a Fr\'echet derivative of $\fob_g$ evaluated at $\gamma$ due to \eqref{fpplimit}, however $\arum$ is not a linear vector space, thus the requirement $\gamma,\gamma+\ds\in\arum$.

The first step in minimizing $\Psi$ using a gradient descent type iterative algorithm is to determine a derivative to the discrepancy terms $R_k$.\ For this purpose the following corollary is applied, and is a special case of \cite[Theorem 3.1]{jin2011} for $\Omega$ being an open and bounded subset of $\rum{R}^3$ with smooth boundary.

\begin{corollary} \label{meyerthm}
	For $\gamma\in\arum$ there exists $Q(c)>2$ depending continuously on the bound $c$ from $\arum$, such that $\lim_{c\to 1}Q(c) = \infty$. For $q\in (2,Q(c))\cap[\frac{3}{2},\frac{3}{2}p]$ and $g\in L^p(\po)\cap\hmhalf$, there is the following estimate with $C$ only depending on $c$, $\Omega$ and $q$:
	\begin{equation}
		\normm{F_g(\gamma)}_{W^{1,q}(\Omega)} \leq C\normm{g}_{L^p(\po)}.
	\end{equation}
\end{corollary}

\begin{lemma} 
	Let $g_k \in L^p(\po)\cap \hmhalf$ with $p\geq \frac{8}{5}$, and $\chi_{\gamD}$ be a characteristic function on $\gamD$. Then there exists $c\in (0,1)$ as the bound in $\arum_0$ sufficiently close to 1, such that $\gamma = \delta\gamma + \sigma_0$ with $\delta\gamma\in\arum_0$ implies
	\begin{equation}
		G_k \equiv -\nabla F_{g_k}(\gamma)\cdot\nabla F_{\chi_{\gamD}(\Lambda_\gamma g_k-f_k)}(\gamma)\in L^{6/5}(\Omega)\subset H^{-1}(\Omega), \label{nablaJ}
	\end{equation}
	and the Fr\'echet derivative $(R_k)'_{\delta\gamma}$ of $R_k$ on $H_0^1(\Omega)$ evaluated at $\delta\gamma$ in the direction $\eta$ is given by
	\begin{equation}
		(R_k)'_{\delta\gamma}\eta = \int_{\Omega} G_k\ds \,dx, \enskip \delta\gamma+\ds\in \arum_0. \label{Jprimeeq}
	\end{equation}
\end{lemma}
\begin{proof}
  For the proof the index $k$ is suppressed. First it is proved that
  $G\in L^{6/5}(\Omega).$ Write $h \equiv \chi_{\gamD}(\Lambda_\gamma
  g-f)$ and note that $\Lambda_\gamma g \in H_{\gamD}^{1/2}(\po)$ and
  $f\in L^2_\diamond(\gamD)$, i.e. $h \in L^2_\diamond(\po) \subset
  L^2(\po)\cap H_\diamond^{-1/2}(\po)$. Now using Corollary
  \ref{meyerthm}, there exists $Q(c)>2$ such that
	\begin{equation}
		\normm{F_h(\gamma)}_{W^{1,q}(\Omega)} \leq C\normm{h}_{L^2(\po)}, \label{Fweq}
	\end{equation}
	where $q\in(2,Q(c))\cap[\frac{3}{2},3]$. Since $g\in L^{8/5}(\po)\cap\hmhalf$ then Corollary ~\ref{meyerthm} implies 
	\begin{equation}
		\normm{F_g(\gamma)}_{W^{1,\tilde{q}}(\Omega)} \leq \tilde{C}\normm{g}_{L^{8/5}(\Omega)}, \label{Fgeq}
	\end{equation}
	for $\tilde{q}\in (2,Q(c))\cap[\frac{3}{2},\frac{12}{5}]$. Choosing $c$ sufficiently close to $1$ leads to $Q(c)> \frac{12}{5}$. By \eqref{Fweq} and \eqref{Fgeq} then $\absm{\nabla F_h(\gamma)},\absm{\nabla F_g(\gamma)} \in L^{12/5}(\Omega)$, and H\"{o}lder's generalized inequality entails that $G\in L^r(\Omega)$ with $\frac{1}{r} = \frac{5}{12}+\frac{5}{12},$ i.e. $r = \frac{6}{5}$,
	\[
		G = -\nabla F_g(\gamma)\cdot \nabla F_h(\gamma) \in L^{6/5}(\Omega).
	\]
	The Sobolev embedding theorem \cite{sobolev} implies the embedding $H^1(\Omega)\imbed L^6(\Omega)$ as $\Omega\subset\rum{R}^3$. Thus $G \in L^{6/5}(\Omega) = (L^6(\Omega))'\subset (H^1(\Omega))' \subset (H_0^1(\Omega))' = H^{-1}(\Omega)$.
	
	Next we prove
        \eqref{Jprimeeq}. % Now it will be shown that $R'_{\delta\gamma}$ can be identified with $G$.
        $R'_{\delta\gamma}\eta$ is by the chain rule (utilizing that
        $\Lambda_\gamma g = \fob_g(\gamma)$) given as
	\begin{equation}
		R'_{\delta\gamma}\eta = \int_\po \chi_{\gamD}(\Lambda_\gamma g-f)(\fob_{g})_\gamma'\ds \,ds, \label{expectedprime}
	\end{equation}
	where $\chi_{\gamD}$ is enforcing that the integral is over $\gamD$. The weak formulations of \eqref{pde}, with Neumann data $\chi_{\gamD}(\Lambda_\gamma g-f)$, and \eqref{fprimepde} are
	\begin{align}
		\int_{\Omega} \gamma \nabla F_{\chi_{\gamD}(\Lambda_\gamma g-f)}(\gamma)\cdot \nabla v\,dx &= \int_{\po} \chi_{\gamD}(\Lambda_\gamma g-f)v|_{\po}\,ds, \enskip \forall v\in H^1(\Omega), \label{weak1} \\
		\int_{\Omega} \gamma \nabla w\cdot \nabla v\,dx &= -\int_{\Omega} \ds\nabla F_{g}(\gamma)\cdot \nabla v\,dx, \enskip \forall v\in H^1(\Omega). \label{weak2}
	\end{align}
	Now by letting $v \equiv w$ in \eqref{weak1} and $v \equiv F_{\chi_{\gamD}(\Lambda_\gamma g-f)}(\gamma)$ in \eqref{weak2}, we obtain using the definition $w|_{\po} = (\fob_g)'_\gamma\eta$ that
	\begin{align*}
		R'_{\delta\gamma}\ds &= \int_\po \chi_{\gamD}(\Lambda_\gamma g-f)(\fob_{g})_\gamma'\ds \,ds = \int_{\Omega} \gamma \nabla F_{\chi_{\gamD}(\Lambda_\gamma g-f)}(\gamma)\cdot \nabla w\,dx \\
		&=  -\int_{\Omega} \ds\nabla F_{g}(\gamma)\cdot \nabla F_{\chi_{\gamD}(\Lambda_\gamma g-f)}(\gamma)\,dx = \int_{\Omega}G\ds \,dx.
	\end{align*}
\end{proof}
\noindent Define 
\[
	R'_{\delta\gamma} \equiv \sum_{k=1}^K (R_k)'_{\delta\gamma} = -\sum_{k=1}^K\nabla F_{g_k}(\gamma)\cdot\nabla F_{\chi_{\gamD}(\Lambda_{\gamma} g_k-f_k)}(\gamma).
\]
%
%For a gradient type descent method, w
We seek to find a direction $\eta$ for which the discrepancy decreases. As $R_{\delta\gamma}'\in H^{-1}(\Omega)$ it is known from Riesz' representation theorem that there exists a unique function in $H_0^1(\Omega)$, denoted by $\nabla_s R(\delta\gamma)$, such that 
\begin{equation}
	R_{\delta\gamma}'\eta = \inner{\nabla_s R(\delta\gamma),\eta}_{H^1(\Omega)},\enskip \eta\in H_0^1(\Omega). \label{sobograd}
\end{equation}
Now $\eta \equiv -\nabla_s R(\delta\gamma)$ points in the steepest descend direction among the viable directions. Furthermore, since $\nabla_s R(\delta\gamma)|_{\po} = 0$ the boundary condition $\dsr|_{\po} = 0$ for the approximation will automatically be fulfilled. Note that $\nabla_s R(\delta\gamma)$ is the unique solution to 
\[
	(-\Delta+1)v=R_{\delta\gamma}' \text{ in } \Omega, \quad v = 0\text{ on } \po,
\]
for which \eqref{sobograd} is the weak formulation. In each iteration step we need to determine a step size $s_i$ for an algorithm resembling a steepest descent $\delta\gamma_{i+1} = \delta\gamma_i - s_i\nabla_sR(\delta\gamma_i)$. As in \cite{GardeKnudsen2014} a Barzilai-Borwein step size rule is applied
\begin{equation}
	s_i = \frac{\normm{\delta\gamma_i-\delta\gamma_{i-1}}_{H^1(\Omega)}^2}{\innerm{\delta\gamma_i-\delta\gamma_{i-1},\nabla_s R(\delta\gamma_i)-\nabla_sR(\delta\gamma_{i-1})}_{H^1(\Omega)}}. \label{bbstepsize}
\end{equation}
A maximum step size $s_{\max}$ is enforced to avoid the situation $\innerm{\delta\gamma_i-\delta\gamma_{i-1},\nabla_s R(\delta\gamma_i)-\nabla R(\delta\gamma_{i-1})}_{H^1(\Omega)}\simeq 0$.

With inspiration from \cite{sparsa}, $s_i$ will be initialized by
\eqref{bbstepsize}, after which it is thres\-holded to lie in
$[s_{\min},s_{\max}]$ for two chosen positive constants $s_{\min}$ and
$s_{\max}$. It is noted in \cite{sparsa} that Barzilai-Borwein type
step rules lead to faster convergence if we do not restrict $\Psi$ to
decrease in every iteration. Therefore, one makes sure that the
following so-called weak monotonicity is satisfied, which compares
$\Psi(\delta\gamma_{i+1})$ with the most recent $M$ steps. Let
$\tau\in(0,1)$ and $M\in\rum{N}$, then $s_i$ is said to satisfy the
weak monotonicity with respect to $M$ and $\tau$ if the following is
satisfied
\begin{equation}
	\Psi(\delta\gamma_{i+1}) \leq \max_{i-M+1\leq j \leq i}\Psi(\delta\gamma_j) - \frac{\tau}{2s_i}\normm{\delta\gamma_{i+1}-\delta\gamma_i}_{H^1(\Omega)}^2.\label{weakmono}
\end{equation}
If \eqref{weakmono} is not satisfied, the step size $s_i$ is reduced until this is the case. 

To solve the non-linear minimization problem we iteratively solve the following linearized problem
\begin{align}
	\zeta_{i+1} &\equiv \argmin_{\delta\gamma\in H_0^1(\Omega)}\left[\frac{1}{2}\normm{\delta\gamma - (\delta\gamma_i-s_i\nabla_sR(\delta\gamma_i))}_{H^1(\Omega)}^2 + s_i\sum_{j=1}^\infty \alpha_j\absm{c_j}\right], \label{upsiloneq} \\
	\delta\gamma_{i+1} &\equiv \proja(\zeta_{i+1}). \notag
\end{align}
Here $\{\psi_j\}$ is an orthonormal basis for $H_0^1(\Omega)$ in the $H^1$-metric, and $\proja$ is a projection of $H_0^1(\Omega)$ onto $\arum_0$ to ensure that \eqref{pde} is solvable (note that $H_0^1(\Omega)$ does not embed into $L^\infty(\Omega)$, i.e. $\zeta_{i+1}$ may be unbounded). By use of the map $\mathcal{S}_\beta:\rum{R}\to\rum{R}$ defined below, known as the soft shrinkage/thresholding map with threshold $\beta > 0$,
\begin{equation}
	\ssc{\beta}(x) \equiv \sign(x)\max\{\absm{x}-\beta,0\},\enskip x\in\rum{R}, \label{softoperator}
\end{equation}
the solution to \eqref{upsiloneq} is easy to find directly (see also \cite[Section 1.5]{daubechies2004})
\begin{equation}
	\zeta_{i+1} = \sum_{j=1}^\infty \ssc{s_i\alpha_j}(d_j)\psi_j, \label{softstep}
\end{equation}
where $d_j\equiv\innerm{\delta\gamma_i-s_i\nabla_sR(\delta\gamma_i),\psi_j}_{H^1(\Omega)}$ are the basis coefficients for $\delta\gamma_i-s_i\nabla_sR(\delta\gamma_i)$. 

The projection $\proja : H_0^1(\Omega)\to \arum_0$ is defined as
\begin{equation*}
	\proja(v) \equiv T_c(\sigma_0 + v) - \sigma_0, \enskip v\in H_0^1(\Omega),
\end{equation*}
where $T_c$ is the following truncation that depends on the constant $c\in(0,1)$ in \eqref{a0ref}
\begin{equation*}
	T_c(v) \equiv \begin{cases}
		c & \text{where } v < c \text{ a.e.}, \\
		c^{-1} & \text{where } v > c^{-1} \text{ a.e.}, \\
		v & \text{else.} 
	\end{cases}
\end{equation*}
Since $\sigma_0\in H^1(\Omega)$ and $c\leq \sigma_0\leq c^{-1}$, it follows directly from \cite[Lemma 1.2]{Stampacchia_1965} that $T_c$ and $\proja$ are well-defined, and it is easy to see that $\proja$ is a projection. It should also be noted that $0\in \arum_0$ since $c\leq \sigma_0\leq c^{-1}$, thus we may choose $\delta\gamma_0 \equiv 0$ as the initial guess in the algorithm.

The algorithm is summarized in Algorithm \ref{alg1}. In the numerical
experiments in Section \ref{sec:NumericalResults} the stopping criteria is
when the step size $s_i$ gets below a threshold $s_\text{stop}$.
%
%\vspace{-15pt}
%
\begin{algorithm} \caption{Sparse Reconstruction for Partial Data EIT} \label{alg1}
\begin{algorithmic}
	\State Set $\delta\gamma_0 := 0$.
	\While{stopping criteria not reached}
		\State Set $\gamma_i := \sigma_0 + \delta\gamma_i$.
		\State Compute $\Psi(\delta\gamma_i)$.
		\State Compute $R_{\delta\gamma_i}' := -\sum_{k=1}^K\nabla F_{g_k}(\gamma_i)\cdot\nabla F_{\chi_{\gamD}(\Lambda_{\gamma_i} g_k-f_k)}(\gamma_i)$.
		\State Compute $\nabla_s R(\delta\gamma_i)\in H_0^1(\Omega)$ such that $R_{\delta\gamma_i}'\ds = \innerm{\nabla_s R(\delta\gamma_i),\ds}_{H^1(\Omega)}$.
		\State Compute step length $s_i$ by \eqref{bbstepsize}, and decrease it till \eqref{weakmono} is satisfied.
		\State Compute the basis coefficients $\{d_j\}_{j=1}^\infty$ for $\delta\gamma_i-s_i\nabla_sR(\delta\gamma_i)$.
		\State Update $\delta\gamma_{i+1} := \proja\left(\sum_{j=1}^\infty \ssc{s_i\alpha_j}(d_j)\psi_j\right)$.	
	\EndWhile
	\State Return final iterate of $\delta\gamma$.
\end{algorithmic}
\end{algorithm}

\section{Prior Information} \label{sec:prior}

Prior information is typically introduced in the penalty term $P$ for Tikhonov-like functionals, and here the regularization parameter determines how much this prior information is enforced. In the case of sparsity regularization this implies knowledge of how sparse we expect the solution is in general. Instead of applying the same prior information for each basis function, a distributed parameter is applied. Let
\[
	\alpha_j \equiv \alpha\mu_j,
\] 
where $\alpha$ is a usual regularization parameter, corresponding to the case where no prior information is considered about specific basis functions. The $\mu_j\in(0,1]$ will be used to weigh the penalty depending on whether a specific basis function should be included in the expansion of $\delta\sigma$. The $\mu_j$ are chosen as
\[
\mu_j = \begin{cases}
	1, \quad & \text{no prior on $c_j$}, \\ \sim 0, & \text{prior that $c_j\neq 0$},	
\end{cases}
\]
i.e. if we know that a coefficient in the expansion of $\dsr$ should be non-zero, we can choose to penalize that coefficient less. Ideally, if we know that a coefficient should be non-zero we would actually choose $\mu_j = 0$, however, in most cases we might only have an estimate of which basis functions that should be included in the solution. Choosing $\mu_j = 0$ will effectively remove any regularization of the corresponding basis function, and may introduce further instability into the numerical algorithm.

\subsection{Applying the FEM Basis}

In order to improve the sparsity solution for finding small
inclusions, it seems appropriate to include prior information about
the support of the inclusions. There are different methods available
for obtaining such information assuming piecewise constant
conductivity \cite{HarrachUllrich2013,kirsch2007} or real analytic
conductivity \cite{HarrachSeo2010}. The idea is to be able to apply such information in the sparsity algorithm in order to get good contrast reconstruction while
maintaining the correct support, even for the partial data problem. 

Suppose that as a basis we consider a finite element method (FEM)
basis $\{\psi_j\}_{j=1}^N$ for the subspace $V_h\subseteq H_0^1(\Omega)$ of piecewise affine
functions on each element. Let $\delta\gamma\in V_h$ with mesh nodes $\{x_j\}_{j=1}^N$, then $\delta\gamma(x) = \sum_{j=1}^N \delta\gamma(x_j)\psi_j(x)$ and $\psi_j(x_k) = \delta_{j,k}$, i.e.\ for each node there is a basis function for which the coefficient
contains local information about the expanded function; this is
convenient when applying prior information about the support of an
inclusion.

When applying the FEM basis for mesh nodes $\{x_j\}_{j=1}^N$, the corresponding functional is 
\[
	\Psi(\delta\gamma) = \frac{1}{2}\sum_{k=1}^K \normm{\Lambda_{\sigma_0+\delta\gamma}g_k - f_k}_{L^2(\gamD)}^2 + \sum_{j=1}^N \alpha_j\absm{\delta\gamma(x_j)}.
\]
It is evident that the penalty corresponds to determining inclusions with small support, and prior information on the sparsity corresponds to prior information on the support of $\delta\sigma$. We cannot directly utilize \eqref{softstep} due to the FEM basis not being an orthonormal basis for $H_0^1(\Omega)$, and instead we suggest the following iteration step as in \cite{GardeKnudsen2014}:
\begin{align}
	\zeta_{i+1}(x_j) &= \ssc{s_i\alpha_j/\normm{\psi_j}_{L^1(\Omega)}}(\delta\gamma_i(x_j)-s_i\nabla_sR(\delta\gamma_i)(x_j)),\enskip j=1,2,\dots,N, \label{femupdate} \\
	\delta\gamma_{i+1} &= \proja(\zeta_{i+1}). \notag
\end{align}
Note that the regularization parameter will depend quite heavily on the discretization of the mesh, i.e. for the same domain a good regularization parameter $\alpha$ will be much larger on a coarse mesh than on a fine mesh. Instead we can weigh the regularization parameter according to the mesh cells, by having $\alpha_j \equiv \alpha\beta_j\mu_j$. This leads to a discretization of a weighted $L^1$-norm penalty term:
\[
	\alpha\int_{\Omega} f_\mu \absm{\delta\gamma}\,dx \simeq \alpha\sum_j \beta_j\mu_j\absm{\delta\gamma(x_j)}, 
\]
where $f_\mu : \Omega\to (0,1]$ is continuous and $f_\mu(x_j) = \mu_j$. The weights $\beta_j$ consists of the node volume computed in 3D as 1/4 of the volume of $\supp\psi_j$ (if using a mesh of tetrahedrons). This corresponds to splitting each cell's volume evenly amongst the nodes, and it will not lead to instability on a regular mesh. This will make the choice of $\alpha$ almost independent of the mesh, and will be used in the numerical examples in the following section.
\begin{remark}
	The corresponding algorithm with the FEM basis is the same as Algorithm \ref{alg1}, except that the update is applied via \eqref{femupdate}.
\end{remark}

\section{Numerical Examples}\label{sec:NumericalResults}

In this section we illustrate, through a few examples, the numerical algorithm implemented by use of the finite element library FEniCS \cite{logg2012a}. First we consider the full data case $\gamD = \gamN = \partial\Omega$ without and with prior information, and then we do the same for the partial data case.

For the following examples $\Omega$ is the unit ball in $\rum{R}^3$. The numerical phantom consists of a background conductivity with value $1$, a smaller ball inclusion with value $2$ centred at $(-0.09,-0.55,0)$ and with radius $0.35$, and two large ellipsoid inclusions with value $0.5$. One ellipsoid is centred at $(-0.55\sin(\tfrac{5}{12}\pi), 0.55\cos(\tfrac{5}{12}\pi), 0)$ and with semi-axes of length $(0.6,0.3,0.3)$. The other ellipsoid is centred at $(0.45\sin(\tfrac{5}{12}\pi),\allowbreak 0.45\cos(\tfrac{5}{12}\pi), 0)$ and with semi-axes of length $(0.7,0.35,0.35)$. The two ellipsoids are rotated respectively $\tfrac{5}{12}\pi$ and $-\tfrac{5}{12}\pi$ about the axis parallel to the Z-axis and through the centre of the ellipsoids; see Figure \ref{fig:3d_phantom}. 

\begin{figure}[htb]
%\begin{center}
\includegraphics[width = 0.6\textwidth]{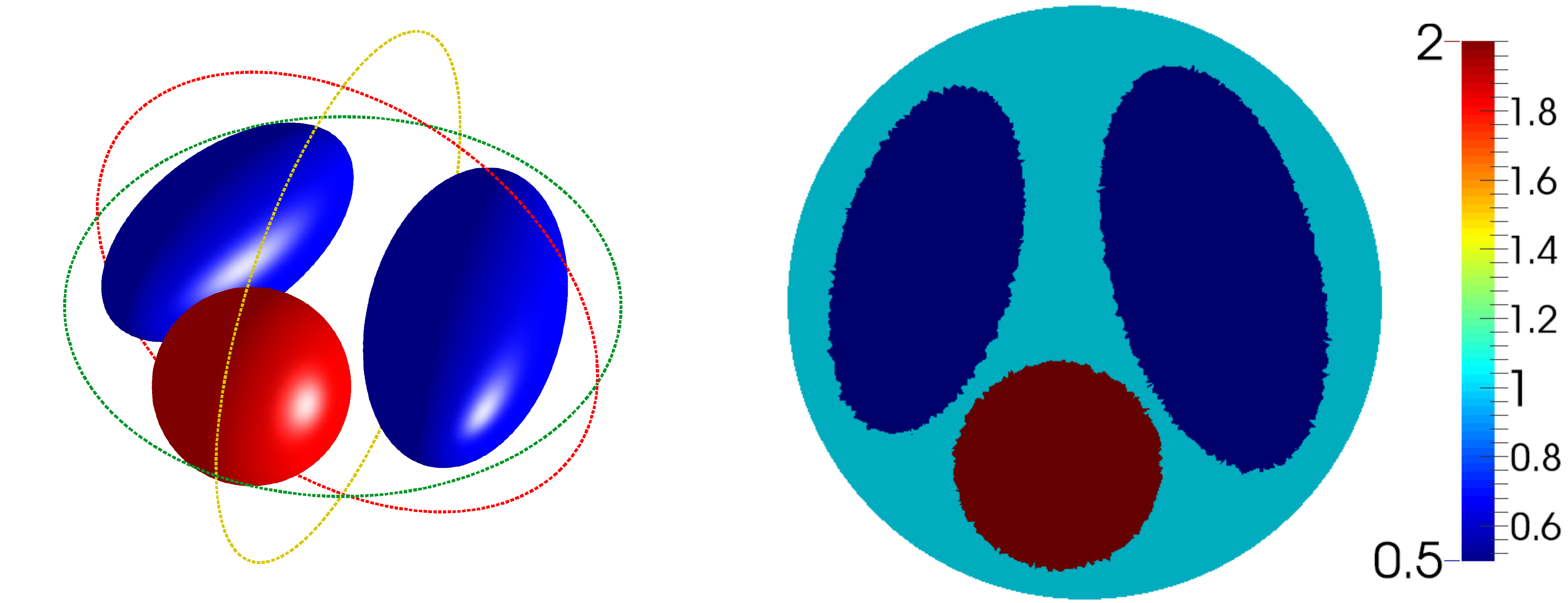}
\caption{\textbf{Left:}\ 3D illustration of the numerical phantom. \textbf{Right}: 2D slice $(z=0)$ of the numerical phantom.} \label{fig:3d_phantom}
%\end{center}
\end{figure}

%\begin{figure}[htb]
%%\begin{center}
%\includegraphics[width = 0.25\textwidth,trim = 3cm 1.3cm 2.5cm 0.7cm, clip=true]{billeder/fantom_sketch}
%\caption{3D illustration of the numerical phantom.} \label{fig:3d_phantom}
%%\end{center}
%\end{figure}

In this paper we do not consider choice rules for $\alpha$; it is chosen manually by trial and error. The parameters are chosen as $\sigma_0\equiv 1$, $M = 5$, $\tau = 10^{-5}$, $s_{\min} = 1$, $s_{\max} = 1000$, and the stopping criteria is when the step size is below $s_{\text{stop}} = 10^{-3}$. Let $Y^m_n$ denote Laplace's spherical harmonics of degree $n$ and order $m$, with real form
\begin{equation}
	\tilde{Y}_n^m = \begin{dcases} \tfrac{i}{\sqrt{2}}(Y^m_n-(-1)^m Y^{-m}_n) & \text{ for } m < 0, \\
	Y_n^0 & \text{ for } m = 0, \\
	\tfrac{1}{\sqrt{2}}(Y^{-m}_n + (-1)^m Y_n^m) & \text{ for } m>0. \end{dcases}
\end{equation}
The Neumann data consists of $\tilde{Y}_n^m$ for $-n\leq m\leq n$ and $n=1,2,\dots,5$, i.e. a total of $K = 35$ current patterns. For the partial data examples a half-sphere is used for local data $\Gamma = \gamN = \gamD$, and the corresponding Neumann data are scaled to have the same number of periods as the full data examples. 

%\begin{figure}[htb]
%%\begin{center}
%\includegraphics[width = .9\textwidth]{billeder/first_block}
%\caption{2D slices $(z=0)$ through centre of ball domain. \textbf{Left:} numerical phantom. \textbf{Middle:} reconstruction with full data and no spatial  prior information. \textbf{Right:} reconstruction with full data and overestimated support as additional prior information.} \label{fig:block_one}
%\end{center}
%\end{figure}

When applying prior information, the coefficients $\mu_j$ are chosen as $10^{-2}$ where the support of $\dsr$ is assumed, and $1$ elsewhere. The assumed support is a $10 \%$ dilation of the true support, to show that this inaccuracy in the prior information still leads to improved reconstructions.

For the simulated Dirichlet data, the forward problem is solved on a very fine mesh, and afterwards interpolated onto a different much coarser mesh in order to avoid inverse crimes. White Gaussian noise has been added to the Dirichlet data $\{f_k\}_{k=1}^K$ on the discrete nodes on the boundary of the mesh. The standard deviation of the noise is chosen as $\epsilon \max_{k}\max_{x_j\in\gamD}\absm{f_k(x_j)}$ as in \cite{GardeKnudsen2014}, where \mbox{$\epsilon = 10^{-2}$} corresponding to 1\% noise.

\begin{figure}[htb]
%\begin{center}
\includegraphics[width = .55\textwidth]{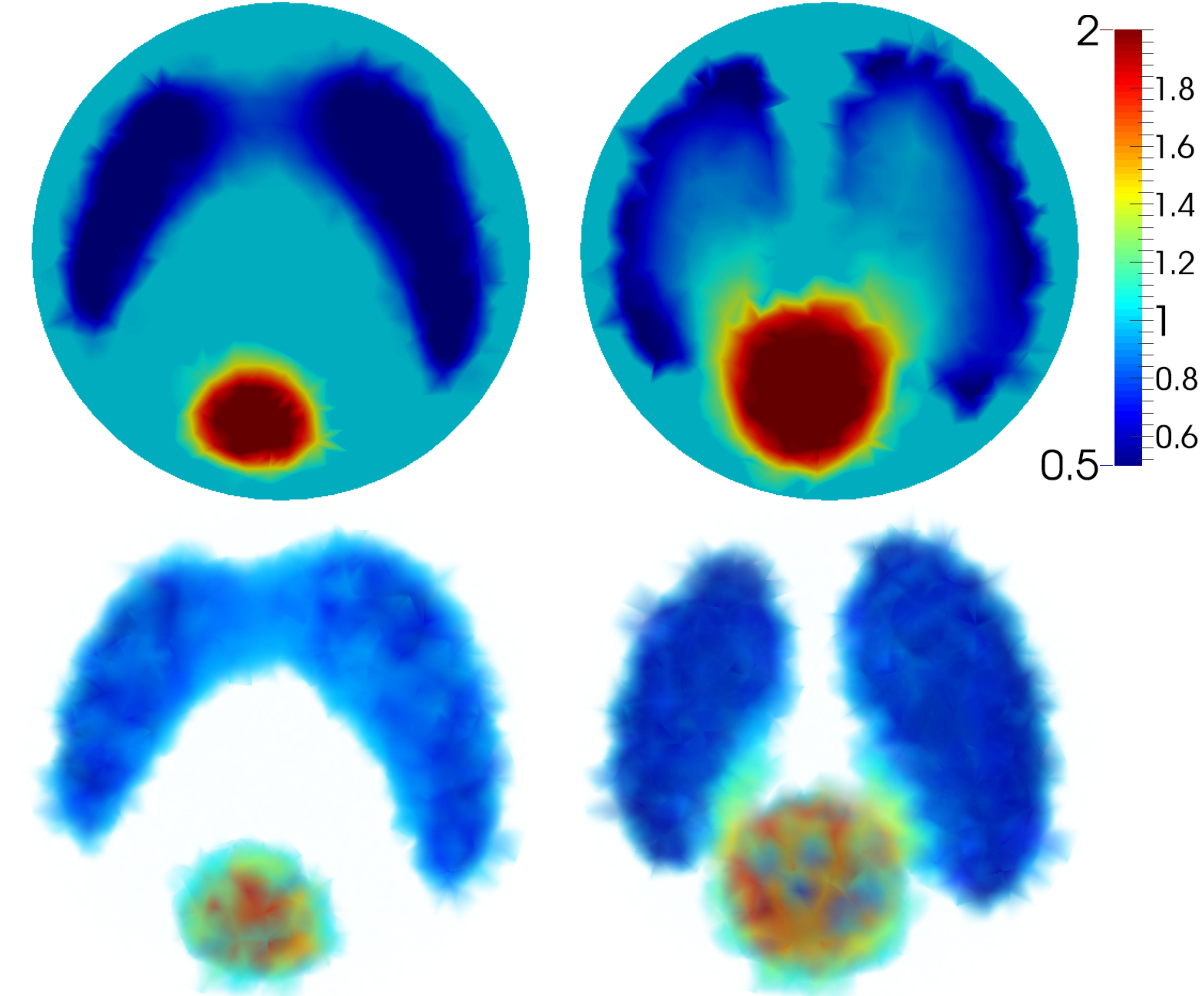}
\caption{\textbf{Top: }2D slices $(z=0)$ through centre of ball domain. \textbf{Bottom:} 3D volume plot where the background value of $1$ is made transparent. \textbf{Left:} reconstruction with full data and no spatial prior information. \textbf{Right:} reconstruction with full data and overestimated support as additional prior information.} \label{fig:block_one}
%\end{center}
\end{figure}

%\begin{figure}[htb]
%%\begin{center}
%\includegraphics[width = .9\textwidth]{billeder/first_block_3d}
%\caption{3D volume plot corresponding to Figure \ref{fig:block_one}. The background value of $1$ is made transparent.} \label{fig:block_one_3d}
%%\end{center}
%\end{figure}

Figure \ref{fig:block_one} shows 2D slices of the numerical phantom and reconstructions from full boundary data. It is seen that the reconstructions attain the correct contrast, and close to the boundary gives good approximations to the correct support for the inclusions. Using the overestimated support as prior information gives vastly improved reconstruction further away from the boundary. This holds for the entire 3D reconstruction as seen in Figure \ref{fig:block_one}, and makes it possible to get a reasonable separation of the inclusions.

\begin{figure}[htb]
%\begin{center}
\includegraphics[width = .85\textwidth]{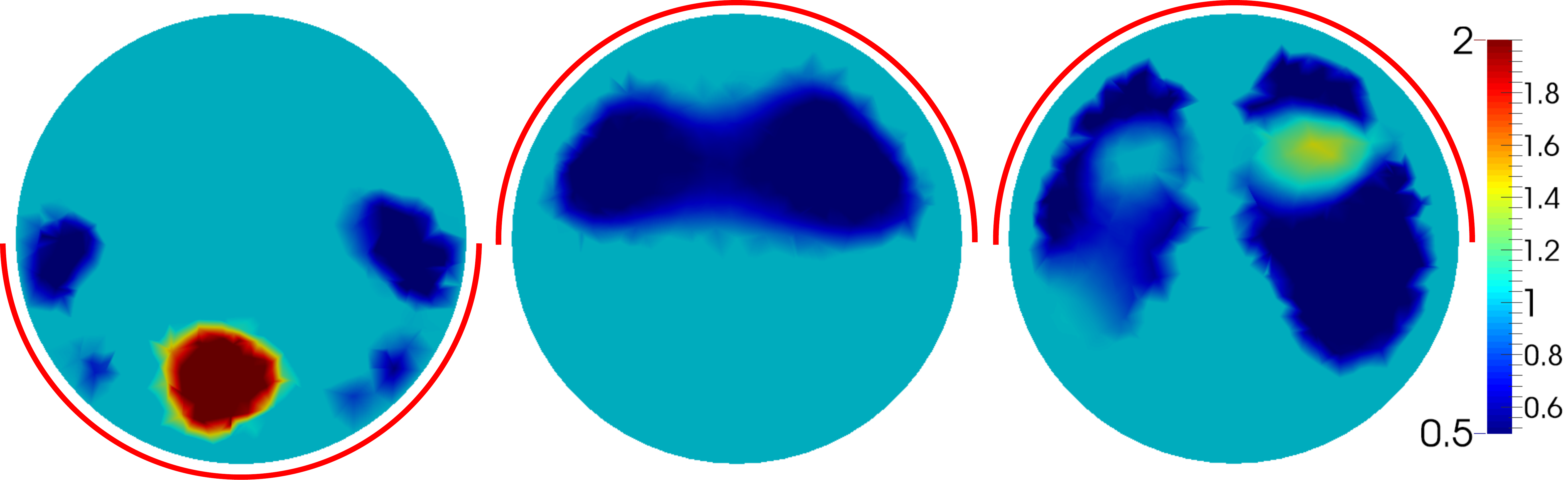}
\caption{2D slices  $(z=0)$ through centre of ball domain \textbf{Left:} reconstruction with data on lower half-sphere and no spatial prior information. \textbf{Middle:} reconstruction with data on upper half-sphere and no spatial prior information. \textbf{Right:} reconstruction with data on upper half-sphere and overestimated support as additional prior information.} \label{fig:block_two}
%\end{center}
\end{figure}

From Figure \ref{fig:block_two} 2D slices of partial data reconstructions are shown, and it is evident that far from the measured boundary the reconstructions suffer severely. Reconstructing with data on the lower part of the sphere gives a reasonable reconstruction with correct contrast for the ball inclusion, however the larger inclusions are hardly reconstructed at all. 

With data on the top half of the sphere yields a reconstruction with no clear separation of the ellipsoid inclusions, which is much improved by use of the overestimated support. There is however an artefact in one of the reconstructed inclusions that could correspond to data from the ball inclusion, which is not detected in the reconstruction even when the additional prior information is used. 

The reconstructions shown here are consistent with what was observed in \cite{GardeKnudsen2014} for the 2D problem, and it is possible to reconstruct the correct contrast even in the partial data case, and also get decent local reconstruction close to the measured boundary. However, the partial data reconstructions seems to be slightly worse in 3D when no prior information about the support is applied.

% You may incorporate your references as follows in your main tex file.
% Using BibTex is not recommended but can be handled.

\end{document}